\documentclass[10pt,A4]{article}
\usepackage[latin5]{inputenc}
\usepackage{amsmath,amssymb}
\usepackage{amsthm}
\usepackage{indentfirst}

\newtheorem{theorem}{Theorem}

\newtheorem{proposition}{Proposition}
\newtheorem{remark}{Remark}
\numberwithin{definition}{section} \numberwithin{theorem}{section}
\numberwithin{lemma}{section}\numberwithin{corollary}{section}
\numberwithin{equation}{section} \numberwithin{example}{section}
\numberwithin{proposition}{section} \numberwithin{remark}{section}
\begin{document}

\begin{center}

{\bf \Large Continuity of the  $L_{p}$ Balls and an Application to Input-Output System Described by the Urysohn Type Integral Operator}

\vspace{3mm}

Anar Huseyin$^1$, Nesir Huseyin$^2$, Khalik G. Guseinov$^3$

\vspace{3mm}
{\small $^1$Cumhuriyet University, Faculty of Science, Department of Statistics and Computer Sciences \\ 58140 Sivas, TURKEY

e-mail: ahuseyin@cumhuriyet.edu.tr
\vspace{2mm}

$^2$Cumhuriyet University, Faculty of Education, Department of Mathematics and Science Education \\ 58140 Sivas, TURKEY

e-mail: nhuseyin@cumhuriyet.edu.tr

\vspace{2mm}

$^3$Eskisehir Technical  University, Faculty of Science, Department of Mathematics \\ 26470 Eskisehir, TURKEY

e-mail: kguseynov@eskisehir.edu.tr}

\end{center}

\textbf{Abstract.} In this paper the continuity of the set valued map $p\rightarrow B_{\Omega,\mathcal{X},p}(r),$ $p\in (1,+\infty),$ is proved where $B_{\Omega,\mathcal{X},p}(r)$ is the closed ball of the space $L_{p}\left(\Omega,\Sigma,\mu; \mathcal{X}\right)$ centered at the origin with radius $r,$ $\left(\Omega,\Sigma,\mu\right)$ is a finite and positive measure space, $\mathcal{X}$ is separable Banach space. An application to  input-output system described by Urysohn type integral operator is discussed.

\vspace{5mm}


\vspace{5mm}


\section{Problem Statement}

To solve some problems arising in theory and applications, it is necessary to define the distance between the subsets of different metric spaces (see, e.g. \cite{kot}, \cite{sor} and references therein). For this aim often the Hausdorff-Gromov distance concept is used which is a generalization of the Hausdorff distance notion (see, e.g. \cite{aub}, \cite{bur}, \cite{fil}). In this paper for definition of the distance between the subsets of the spaces $L_p,$ $p> 1,$ the metric of the space $L_1$ is used. It turns out to be possible since $L_p \subset L_1$ for every $p\in (1,+\infty).$ Using the introduced metric the continuity of the closed balls of the spaces $L_p,$ $p> 1,$ with respect to $p$ is established.

Let $(\Omega, \Sigma, \mu)$ be a finite and positive measure space, $\left(\mathcal{X}, \left\|\cdot \right\|\right)$ be a separable Banach space. \linebreak $L_p\left(\Omega, \Sigma, \mu; \mathcal{X}\right)$ stands for the space of all (equivalence classes of) $\mu$-measurable functions $x(\cdot):\Omega \rightarrow \mathcal{X}$ such that  $\left\| x(\cdot) \right\|_p < +\infty$ where $\left\| x(\cdot) \right\|_p=\left(\int_{\Omega}\left\| x(s)\right\|^p \mu (ds) \right)^{\frac{1}{p}},$ integration is understood in the sense of Bochner.

For given $p\geq 1$ and $r>0$ we denote
\begin{eqnarray}\label{bep}
B_{\Omega, \mathcal{X},p}(r)=\left\{ x(\cdot)\in L_p \left(\Omega, \Sigma,\mu; \mathcal{X}\right): \left\| x(\cdot) \right\|_p \leq r \right\}.
\end{eqnarray}

Let $Y\subset L_{p_1} \left(\Omega, \Sigma,\mu; \mathcal{X}\right)$ and $W\subset L_{p_2} \left(\Omega, \Sigma,\mu; \mathcal{X}\right)$ are bounded sets where $p_1 \geq 1$, $p_2 \geq 1.$ The Hausdorff distance between the sets $Y$ and $W$ is denoted by symbol $\mathcal{H}_1(Y,W)$ and defined as
\begin{eqnarray}\label{haus}
\mathcal{H}_1(Y,W)= \max \left\{ \sup_{y(\cdot)\in Y} d_1\left(y(\cdot),W\right), \sup_{w(\cdot)\in W} d_1\left(w(\cdot),Y\right) \right\}
\end{eqnarray} where $d_1\left(y(\cdot),W\right)=\inf \left\{ \left\| y(\cdot)-w(\cdot) \right\|_1 :w(\cdot)\in W\right\}.$

The continuity of the set valued map $p \rightarrow B_{\Omega, \mathcal{X},p}(r):(1,+\infty)\rightarrow L_p \left(\Omega, \Sigma,\mu; \mathcal{X}\right)$ with respect to the pseudometric $\mathcal{H}_1(\cdot,\cdot)$ is studied. The paper is organized as follows:
In Section 2 an auxiliary proposition is proved which is used in following arguments (Proposition \ref{prop2.1}). In Section 3 the lower semicontinuity, in Section 4 the upper semicontinuty of the set valued $p\rightarrow B_{\Omega, \mathcal{X},p}(r)$, $p\in (1,+\infty)$ (Theorem \ref{teo3.1} and Theorem \ref{teo4.1} respectively) is shown. In Section 5 the main result of the paper, the continuity of the considered set valued map is formulated (Theorem \ref{teo5.1}). In Section 6 continuity of the set of outputs of the input-output system described by Urysohn type integral operator is discussed (Theorem \ref{teo6.1}).

\section{Preliminaries}

Let $\alpha >0$ be a given number. We set
\begin{eqnarray}\label{bepal}
B_{\Omega,\mathcal{X}, p}^{\alpha}(r)=\left\{ x(\cdot)\in B_{\Omega,\mathcal{X},p}(r): \left\| x(s) \right\| \leq \alpha \ \mbox{for every} \ s\in \Omega \right\}.
\end{eqnarray}

For given $p_*>1$ and $\varepsilon >0$ we denote
\begin{eqnarray}\label{ce*}
c_*=\max \left\{\max \left\{ r^{\frac{p-p_*}{p}}: p\in \left[p_*,2p_*\right]\right\},\max \left\{ r^{\frac{p_*-p}{p_*}}: p\in \left[\frac{p_*+1}{2},p_*\right]\right\}\right\}
\end{eqnarray}
\begin{eqnarray}\label{de*}
d_*= \left( \mu(\Omega)+1\right)\left(c_*+1\right) ,
\end{eqnarray}
\begin{eqnarray}\label{alep}
\alpha_* (\varepsilon)= \max \left\{2r, \ r \cdot \left(\frac{8r}{\varepsilon}\right)^{\frac{2}{p_*-1}}\right\} \, ,
\end{eqnarray}
\begin{eqnarray}\label{sigma*}
\sigma_* = \min \left\{ \frac{r}{2},4d_*, 4r\mu(\Omega), 4d_*\sqrt{r}\right\}.
\end{eqnarray}

\begin{proposition}\label{prop2.1} Let $p_*>1$ and $\varepsilon \in (0,\sigma_*)$ be fixed numbers. Then for every $\alpha \geq \alpha_* (\varepsilon)$ the inequality
\begin{eqnarray*}
\mathcal{H}_1\left(B_{\Omega,\mathcal{X},p}(r), B_{\Omega,\mathcal{X},p}^{\alpha}(r)\right) \leq \frac{\varepsilon}{4}
\end{eqnarray*}is satisfied for every $p\in \left[\frac{p_*+1}{2}, 2p_*\right]$ where $\alpha_* (\varepsilon)$ is defined by (\ref{alep}), $B_{\Omega,\mathcal{X},p}(r)$ and $B_{\Omega,\mathcal{X},p}^{\alpha}(r)$ are defined by (\ref{bep}) and (\ref{bepal}) respectively, $\mathcal{H}_1(\cdot,\cdot)$ is defined by (\ref{haus}).
\end{proposition}
\begin{proof}
Let $p\in \left[\frac{p_*+1}{2}, 2p_*\right],$ $\alpha >0$ be arbitrarily fixed numbers. Choose an arbitrary $x(\cdot)\in B_{\Omega,\mathcal{X},p}(r)$ and define new function $x_*(\cdot):\Omega \rightarrow \mathcal{X}$ setting
\begin{eqnarray}\label{eq21}
x_*(s)=\left\{
\begin{array}{llll}
x(s) \, , & \mbox{if} &  \left\| x(s) \right\| \leq \alpha \, ,  \\
\displaystyle \frac{x(s)}{\left\|x(s)\right\|}\alpha \, , & \mbox{if} & \left\|x(s)\right\| > \alpha \, .
\end{array}
\right.
\end{eqnarray} It is not difficult to verify that $x_*(\cdot)\in B_{\Omega,p}^{\alpha}(r).$ Denote $E_*=\left\{s\in \Omega: \left\| x(s) \right\| > \alpha  \right\}.$ From the inequality
\begin{eqnarray*}
\alpha^p\mu(E_*) \leq  \int_{E_*} \left\|x(s)\right\|^p \mu(ds)  \leq \int_{\Omega} \left\|x(s)\right\|^p \mu(ds) \leq r^p
\end{eqnarray*}
we have
\begin{eqnarray}\label{eq22}
\mu(E_*) \leq \frac{r^p}{\alpha^p}
\end{eqnarray}
which  is the Chebyshev's inequality for function $x(\cdot)\in B_{\Omega,\mathcal{X},p}(r)$ (see, \cite{whe}, p.82). (\ref{eq21}) and (\ref{eq22}) yield
\begin{eqnarray*}
\left\|x(\cdot)-x_*(\cdot)\right\|_1  \leq  \left[\mu(E_*)\right]^{\frac{p-1}{p}} \left(\int_{E_*} \left\|x(s)-x_*(s)\right\|^p \mu(ds)\right)^{\frac{1}{p}} \leq \frac{2r^{p}}{\alpha^{p-1}}.
\end{eqnarray*} Thus we have
\begin{eqnarray*}
B_{\Omega,\mathcal{X},p}(r) \subset B_{\Omega,\mathcal{X},p}^{\alpha}(r) + \frac{2r^p}{\alpha^{p-1}}B_{\Omega,\mathcal{X},1}(1).
\end{eqnarray*} Since $B_{\Omega,\mathcal{X},p}^{\alpha}(r)\subset B_{\Omega,\mathcal{X},p}(r)$, then the last inclusion yields
\begin{eqnarray}\label{eq23}
\mathcal{H}_1 \left(B_{\Omega,\mathcal{X},p}(r), B_{\Omega,\mathcal{X},p}^{\alpha}(r)\right) \leq \frac{2r^p}{\alpha^{p-1}}.
\end{eqnarray}
where $p\in \left[\frac{p_*+1}{2}, 2p_*\right].$

From inequalities $p\geq \frac{p_*+1}{2}$ and $p_* >1$ it follows that $\displaystyle \frac{2(p-1)}{p_*-1} \geq 1.$ Since $\displaystyle 0<\varepsilon <\sigma_* \leq \frac{r}{2},$ $\alpha \geq \alpha_*(\varepsilon) \geq 2r,$ $\alpha_*(\varepsilon) \geq r \cdot \left(\frac{8r}{\varepsilon}\right)^{\frac{2}{p_*-1}},$ then by virtue of (\ref{eq23}) we have that for every $\alpha \geq \alpha_*(\varepsilon)$ and $p\in \left[\frac{p_*+1}{2}, 2p_*\right]$ the inequality
\begin{eqnarray*}
\mathcal{H}_1 \left(B_{\Omega,\mathcal{X},p}(r), B_{\Omega,\mathcal{X},p}^{\alpha}(r)\right) &\leq &  2r\cdot \left(\frac{r}{\alpha_*(\varepsilon)}\right)^{p-1} \leq
2r \left(\frac{r}{r\cdot \left(\frac{8r}{\varepsilon}\right)^{\frac{2}{p_*-1}}}\right)^{p-1} \\ & = & 2r\cdot \left( \frac{\varepsilon}{8r}\right)^{\frac{2(p-1)}{p_*-1}} \leq 2r \cdot \frac{\varepsilon}{8r} =\frac{\varepsilon}{4}
\end{eqnarray*}
is held.
\end{proof}

\section{Lower Semicontinuty}

In this section it will be proved that the set valued map $p\rightarrow B_{\Omega,\mathcal{X},p}(r)$, $p>1,$ is lower semicontinuous.

\begin{proposition}\label{prop3.1}
Assume $p_*>1,$ $\varepsilon \in \left(0, \sigma_*\right).$ Then there exists $\delta_1(\varepsilon) \in\left(0, \frac{p_*-1}{2}\right]$ such that for every $p\in \left(p_*-\delta_1(\varepsilon), p_*\right)$ the inclusion
\begin{eqnarray*}
B_{\Omega,\mathcal{X},p_*}(r) \subset B_{\Omega,\mathcal{X},p}(r) +\varepsilon B_{\Omega,\mathcal{X},1}(1)
\end{eqnarray*} is satisfied where $\sigma_* >0$ is defined by (\ref{sigma*}).
\end{proposition}

\begin{proof}
At first let us prove that for $\varepsilon \in \left(0, \sigma_*\right)$ and $\alpha_2 >\alpha_1 >\alpha_*(\varepsilon)$ there exists $\gamma_1(\varepsilon,\alpha_1,\alpha_2) \in\left(0, \frac{p_*-1}{2}\right]$ such that for every $p\in \left(p_*-\gamma_1, p_*\right)$ the inclusion
\begin{eqnarray}\label{eq32a}
B_{\Omega,\mathcal{X},p_*}^{\alpha_1}(r) \subset B_{\Omega,\mathcal{X},p}^{\alpha_2}(r) +\frac{\varepsilon}{2} B_{\Omega,\mathcal{X},1}(1)
\end{eqnarray} is satisfied where $\alpha_*(\varepsilon)>0$ is defined by (\ref{alep}).

Denote
\begin{eqnarray}\label{eq32}
\gamma_1(\varepsilon,\alpha_1,\alpha_2)=\min \left\{ \sigma_1 (\varepsilon, \alpha_1), \sigma_2 (\varepsilon, \alpha_1), \sigma_3 (\alpha_1, \alpha_2), \frac{p_*-1}{2} \right\}
\end{eqnarray}
where
\begin{eqnarray}\label{eq34}
\sigma_1(\varepsilon,\alpha_1)=p_* \left[1-  \frac{1}{1+\log_{\displaystyle \frac{\varepsilon}{4r \mu(\Omega)}} \left(1- \displaystyle \frac{\varepsilon}{4\alpha_1 \mu(\Omega)}\right)} \right],
\end{eqnarray}
\begin{eqnarray}\label{eq35}
\sigma_2(\varepsilon,\alpha_1)=p_* \left[1-  \frac{1}{1+\log_{\displaystyle \frac{\alpha_1}{r}}\left( 1+ \displaystyle \frac{\varepsilon}{4\alpha_1\mu(\Omega)}\right)} \right],
\end{eqnarray}
\begin{eqnarray}\label{eq35a}
\sigma_3(\alpha_1,\alpha_2)=p_* \left[1-  \frac{1}{1+\log_{\displaystyle \frac{\alpha_1}{r}} \displaystyle \frac{\alpha_2}{\alpha_1}} \right].
\end{eqnarray}

Taking into consideration that $p_*>1,$ $\alpha_2 >\alpha_1 >\alpha_*(\varepsilon)\geq 2r > r,$ $\varepsilon <4r\mu(\Omega)<4\alpha_1 \mu(\Omega)$, we obtain that $\sigma_1(\varepsilon, \alpha_1)>0,$ $\sigma_2(\varepsilon, \alpha_1)>0,$ $\sigma_3(\alpha_1, \alpha_2)>0$ and hence $\gamma_1=\gamma_1(\varepsilon,\alpha_1,\alpha_2) \in \left( 0,\frac{p_*-1}{2} \right].$

Choose an arbitrary $p\in \left(p_*-\gamma_1, p_*\right)$ and $x_0(\cdot)\in B_{\Omega,\mathcal{X},p_*}^{\alpha_1}(r)$ and define new function $x_p(\cdot):\Omega \rightarrow \mathcal{X}$ setting
\begin{eqnarray}\label{eq36}
x_p(s)= x_0(s) \left\| x_0(s)\right\|^{\frac{p_*-p}{p}} r^{\frac{p-p_*}{p}}, \ \ s\in \Omega.
\end{eqnarray}

One can show that $\left\|x_p(\cdot)\right\|_{p}\leq r.$ From inclusion $p\in \left(p_*-\gamma_1, p_*\right)$, (\ref{eq32}) and (\ref{eq35a}) we have
\begin{eqnarray*}
p>p_*-\gamma_1 \geq p_* -\sigma_3(\alpha_1,\alpha_2) = \frac{p_*}{\log_{\frac{\alpha_1}{r}}\frac{\alpha_2}{r}}.
\end{eqnarray*}
Taking into consideration that  $\alpha_2 >\alpha_1 >\alpha_*(\varepsilon)\geq 2r$, from the last inequality we obtain  that
\begin{eqnarray}\label{eq39}
\left(\frac{\alpha_1}{r}\right)^{\frac{p_*}{p}}<\frac{\alpha_2}{r} \, .
\end{eqnarray}

Now inclusions $p\in \left(p_*-\gamma_1, p_*\right),$ $x_0(\cdot)\in B_{\Omega,\mathcal{X},p_*}^{\alpha_1}(r)$, (\ref{eq36}) and (\ref{eq39}) imply that
\begin{eqnarray*}
\left\|x_p(s)\right\|&=&  \left\| x_0(s)\right\|^{\frac{p_*}{p}} r^{1-\frac{p_*}{p}}  \leq
r \cdot \left(\frac{\alpha_{1}}{r}\right)^{\frac{p_*}{p}}< r\cdot \frac{\alpha_2}{r} = \alpha_2
\end{eqnarray*}
for every $s\in \Omega.$
Thus, we conclude that $x_p(\cdot)\in B_{\Omega,\mathcal{X},p}^{\alpha_2}(r).$

(\ref{eq36}) implies that
\begin{eqnarray}\label{eq311}
\left\|x_0(\cdot)-x_p(\cdot)\right\|_1 =\int_{\Omega} \left\|x_0(s)\right\|\cdot \left|1- \left(\frac{\left\|x_0(s)\right\|}{r}\right)^{\frac{p_* -p}{p}} \right| \mu(ds) .
\end{eqnarray}

Denote
\begin{eqnarray}\label{eq312}
\Omega_*(\varepsilon)=\left\{ s\in \Omega: 0\leq \left\| x_0(s) \right\| \leq  \frac{\varepsilon}{4\mu(\Omega)}\right\}, \  \ \Omega^*(\varepsilon)=\left\{ s \in \Omega: \frac{\varepsilon}{4\mu(\Omega)} < \left\| x_0(s) \right\| \leq  \alpha_1\right\}.
\end{eqnarray}

Choose an arbitrary $s\in \Omega_*(\varepsilon).$ Since $\varepsilon \in (0,\sigma_*),$ $\frac{p_*-p}{p}>0,$ then from (\ref{sigma*}) and (\ref{eq312}) it follows that
\begin{eqnarray*}
0\leq 1- \left(\frac{\left\|x_0(s)\right\|}{r}\right)^{\frac{p_*-p}{p}} \leq 1
\end{eqnarray*}
and consequently
\begin{eqnarray}\label{eq314}
\int_{\Omega_*(\varepsilon)} \left\|x_0(s)\right\|\cdot \left|1- \left(\frac{\left\|x_0(s)\right\|}{r}\right)^{\frac{p_* -p}{p}}\right| \mu(ds) \leq \frac{\varepsilon}{4\mu(\Omega)} \cdot \mu(\Omega_*(\varepsilon))\leq \frac{\varepsilon}{4}.
\end{eqnarray}

Now let $s\in \Omega^*(\varepsilon).$ Then (\ref{eq312}) implies that
\begin{eqnarray}\label{eq315}
1- \left(\frac{\alpha_1}{r}\right)^{\frac{p_*-p}{p}}\leq 1- \left(\frac{\left\|x_0(s)\right\|}{r}\right)^{\frac{p_*-p}{p}} \leq 1- \left(\frac{\varepsilon}{4\mu (\Omega)r}\right)^{\frac{p_*-p}{p}}.
\end{eqnarray}

Since $\frac{\varepsilon}{4r  \mu(\Omega)}<1$ and for chosen $p\in (p_*-\gamma_1,p_*)$ we have $p_*-p<\gamma_1 \leq \sigma_1(\varepsilon, \alpha_1),$ then from (\ref{eq34}) it follows that
\begin{eqnarray}\label{eq316}
1- \left(\frac{\varepsilon}{4\mu (\Omega)r}\right)^{\frac{p_*-p}{p}} < \frac{\varepsilon}{4\mu(\Omega)\alpha_1}.
\end{eqnarray}

Similarly, from inequalities $\alpha_1 >\alpha_*(\varepsilon)\geq 2r >r,$ $p_*-p<\gamma_1 \leq \sigma_2(\varepsilon, \alpha_1)$ and (\ref{eq35}) we obtain that
\begin{eqnarray}\label{eq316*}
-\frac{\varepsilon}{4\mu(\Omega)\alpha_1} <  1- \left(\frac{\alpha_1}{r}\right)^{\frac{p_*-p}{p}}
\end{eqnarray} is satisfied.

Taking into consideration (\ref{eq312}), (\ref{eq315}), (\ref{eq316}) and (\ref{eq316*}) we have that for $p\in (p_*-\gamma_1,p_*)$ and $s\in \Omega^*(\varepsilon)$ the inequality
\begin{eqnarray}\label{eq318}
\int_{\Omega^*(\varepsilon)} \left\|x_0(s)\right\|\cdot \left|1- \left(\frac{\left\|x_0(s)\right\|}{r}\right)^{\frac{p_* -p}{p}}\right| \mu(ds) \leq \alpha_1 \cdot \frac{\varepsilon}{4\alpha_1 \mu(\Omega)} \cdot \mu(\Omega_*(\varepsilon))\leq \frac{\varepsilon}{4}.
\end{eqnarray} is verified.

Now, (\ref{eq311}), (\ref{eq312}), (\ref{eq314}) and (\ref{eq318}) yield
\begin{eqnarray}\label{eq319}
\left\|x_0(\cdot)-x_p(\cdot)\right\|_1 \leq  \frac{\varepsilon}{4}+\frac{\varepsilon}{4} =\frac{\varepsilon}{2}.
\end{eqnarray}

Since $p\in \left(p_*-\gamma_1, p_*\right)$ and $x_0(\cdot)\in B_{\Omega,\mathcal{X},p_*}^{\alpha_1}(r)$ are arbitrarily chosen and $x_p(\cdot)\in B_{\Omega,\mathcal{X},p_*}^{\alpha_2}(r),$ the inequality (\ref{eq319}) gives the validity of the inclusion  (\ref{eq32a}).

Let $\alpha_1(\varepsilon)=2\alpha_*(\varepsilon),$ $\alpha_2(\varepsilon)=3\alpha_*(\varepsilon),$ $\delta_1(\varepsilon)=\gamma_1(\varepsilon,\alpha_1(\varepsilon),\alpha_2(\varepsilon))$ where $\gamma_1(\varepsilon,\alpha_1(\varepsilon),\alpha_2(\varepsilon))$ is defined by (\ref{eq32}) as $\alpha_1=\alpha_1(\varepsilon)$, $\alpha_2=\alpha_2(\varepsilon)$ and $\alpha_*(\varepsilon)>0$ is defined  by (\ref{alep}).
Then $\delta_1(\varepsilon)\in \left(0,\frac{p_*-1}{2}\right]$ and according to  (\ref{eq32a}), for every $p\in \left(p_*-\delta_1(\varepsilon), p_*\right)$ the inclusion
\begin{eqnarray}\label{eq320}
B_{\Omega,\mathcal{X},p_*}^{\alpha_1(\varepsilon)}(r) \subset B_{\Omega,\mathcal{X},p}^{\alpha_2(\varepsilon)}(r) +\frac{\varepsilon}{2} B_{\Omega,\mathcal{X},1}(1)
\end{eqnarray} is held.

By virtue of Proposition \ref{prop2.1} for $\alpha_1(\varepsilon)$ and $\alpha_2(\varepsilon)$ the inclusions
\begin{eqnarray}\label{eq321}
B_{\Omega,\mathcal{X},p}(r) \subset B_{\Omega,\mathcal{X},p}^{\alpha_1(\varepsilon)}(r) +\frac{\varepsilon}{4} B_{\Omega,\mathcal{X},1}(1),  \ \   B_{\Omega,\mathcal{X},p}^{\alpha_2(\varepsilon)}(r) \subset B_{\Omega,\mathcal{X},p}(r) +\frac{\varepsilon}{4} B_{\Omega,\mathcal{X},1}(1) \end{eqnarray} are satisfied for every $p\in \left[\frac{p_*+1}{2},2p_*\right].$

Since $\delta_1(\varepsilon) \leq \frac{p_*-1}{2},$ then  (\ref{eq320}) and (\ref{eq321}) imply that
\begin{eqnarray*}
B_{\Omega,\mathcal{X},p_*}(r) \subset B_{\Omega,\mathcal{X},p_*}^{\alpha_1(\varepsilon)}(r) +\frac{\varepsilon}{4} B_{\Omega,\mathcal{X},1}(1) \subset B_{\Omega,\mathcal{X},p}^{\alpha_2(\varepsilon)}(r) +\frac{3\varepsilon}{4} B_{\Omega,\mathcal{X},1}(1) \subset B_{\Omega,\mathcal{X},p}(r) +\varepsilon B_{\Omega,\mathcal{X},1}(1)
\end{eqnarray*} for every $p\in \left(p_*-\delta_1(\varepsilon), p_*\right).$
\end{proof}

\begin{proposition}\label{prop3.2}
Assume $p_*>1,$ $\varepsilon \in \left(0, \sigma_*\right).$ Then there exists $\delta_2(\varepsilon) \in\left(0, p_*\right]$ such that for every $p\in \left(p_*,p_*+\delta_2(\varepsilon)\right)$ the inclusion
\begin{eqnarray*}
B_{\Omega,\mathcal{X},p_*}(r) \subset B_{\Omega,\mathcal{X},p}(r) +\varepsilon B_{\Omega,\mathcal{X},1}(1)
\end{eqnarray*} is verified where $\sigma_* >0$ is defined by (\ref{sigma*}).
\end{proposition}

\begin{proof}
At first step it will be proved that for $\varepsilon \in \left(0, \sigma_*\right)$ and $\alpha >\alpha_*(\varepsilon)$  there exists $\gamma_2=\gamma_2(\varepsilon,\alpha) \in\left(0, p_*\right]$ such that for every $p\in \left(p_*, p_*+\gamma_2\right)$ the inclusion
\begin{eqnarray} \label{eq322a}
B_{\Omega,\mathcal{X},p_*}^{\alpha}(r) \subset B_{\Omega,\mathcal{X},p}^{\alpha}(r) +\frac{\varepsilon}{2} B_{\Omega,\mathcal{X},1}(1)
\end{eqnarray} is satisfied where $\alpha_*(\varepsilon)>0$ is defined by (\ref{alep}).

We set
\begin{eqnarray}\label{eq323}
\gamma_2(\varepsilon,\alpha)=\min \left\{ \sigma_4 (\varepsilon, \alpha), \sigma_5 (\varepsilon, \alpha), p_* \right\}
\end{eqnarray}
where
\begin{eqnarray}\label{eq324}
\sigma_4(\varepsilon,\alpha)=p_* \left[\frac{1}{1-\log_{\displaystyle \frac{r}{\alpha}} \displaystyle \left(1-\frac{\varepsilon}{4\alpha \mu(\Omega)}\right)} -1 \right],
\end{eqnarray}
\begin{eqnarray}\label{eq325}
\sigma_5(\varepsilon,\alpha)=p_* \left[\frac{1}{1-\log_{\displaystyle \frac{64rd_*^2}{ \varepsilon^2}} \displaystyle \left(1+ \frac{\varepsilon}{4\alpha\mu(\Omega)}\right)}-1 \right],
\end{eqnarray} where $d_*$ is defined by (\ref{de*}).

Taking into consideration (\ref{alep}), (\ref{sigma*}), the inequalities $\alpha >\alpha_*(\varepsilon)\geq 2r>r,$ $\varepsilon <\sigma_*$ one can verify that
\begin{eqnarray*}
0<\log_{\displaystyle \frac{r}{\alpha}}\left(1-\frac{\varepsilon}{4\alpha \mu(\Omega)}\right) <1,
\end{eqnarray*}
\begin{eqnarray*}
0<\log_{\displaystyle \frac{64rd_*^2}{\varepsilon^2}} \left(\displaystyle 1+\frac{\varepsilon}{4\alpha \mu(\Omega)}\right) <1
\end{eqnarray*} which imply that $\sigma_4(\varepsilon, \alpha)>0,$ $\sigma_5(\varepsilon, \alpha)>0.$ Finally, according to (\ref{eq323}) we obtain that  $\gamma_2=\gamma_2(\varepsilon, \alpha)\in \left(0,p_*\right].$

Choose an arbitrary $p\in \left(p_*, p_*+\gamma_2\right)$ and $y_0(\cdot)\in B_{\Omega,\mathcal{X},p_*}^{\alpha}(r)$ and define new function $y_p(\cdot):\Omega \rightarrow \mathcal{X}$ setting
\begin{eqnarray}\label{eq326}
y_p(s)= y_0(s) \left\| y_0(s)\right\|^{\frac{p_*-p}{p}} r^{\frac{p-p_*}{p}}, \ \ s\in \Omega.
\end{eqnarray}

From (\ref{eq326}) and inclusion $y_0(\cdot)\in B_{\Omega,\mathcal{X},p_*}^{\alpha}(r)$  it follows that $\left\|y_p(\cdot)\right\|_{p}\leq r.$
Since $p>p_*,$ $\alpha >2r$ and $y_0(\cdot)\in B_{\Omega,\mathcal{X},p_*}^{\alpha}(r)$, then (\ref{eq326}) implies that
\begin{eqnarray*}
\left\| y_p(s)\right\|= r \left(\frac{\left\|y_0(s)\right\|}{r}\right)^{\frac{p_*}{p}} \leq r \left(\frac{\alpha}{r}\right)^{\frac{p_*}{p}}<r \frac{\alpha}{r}=\alpha
\end{eqnarray*} for every $s\in \Omega$ which yields that  $y_p(\cdot)\in B_{\Omega,\mathcal{X},p}^{\alpha}(r).$

Denote
\begin{eqnarray}\label{eq330}
A_*(\varepsilon)=\left\{ s\in \Omega: 0\leq \left\| y_0(s) \right\| \leq  \left(\frac{\varepsilon}{8d_{*}}\right)^2\right\},
\end{eqnarray}
\begin{eqnarray}\label{eq331}
 A^*(\varepsilon)=\left\{ s \in \Omega: \left(\frac{\varepsilon}{8d_{*}}\right)^{2} < \left\| y_0(s) \right\| \leq  \alpha\right\}
\end{eqnarray} where $d_*$ is defined by (\ref{de*}).

It is obvious that
\begin{eqnarray}\label{eq332}
\left\|y_0(\cdot)-y_p(\cdot)\right\|_1 &=&\int_{A_*(\varepsilon)} \left\|y_0(s)- y_0(s)\left\|y_0(s)\right\|^{\frac{p_* -p}{p}} r^{\frac{p - p_*}{p}}\right\| \mu(ds)  \nonumber \\
&+& \int_{A^*(\varepsilon)} \left\|y_0(s)- y_0(s)\left\|y_0(s)\right\|^{\frac{p_* -p}{p}} r^{\frac{p - p_*}{p}}\right\| \mu(ds).
\end{eqnarray}

Since $\varepsilon \in (0,\sigma_*),$  $p\in (p_*,p_*+\gamma_2),$ we have $\frac{1}{2}<\frac{p_*}{p}<1,$ $\varepsilon < 8d_*.$ Thus, from (\ref{ce*}), (\ref{de*}), (\ref{eq330}) we obtain
\begin{eqnarray}\label{eq332a}
&& \int_{A_*(\varepsilon)} \left\|y_0(s)- y_0(s)\left\|y_0(s)\right\|^{\frac{p_* -p}{p}} r^{\frac{p - p_*}{p}}\right\| \mu(ds) \nonumber \\ && \leq \int_{A_*(\varepsilon)} \left\|y_0(s)\right\| \mu(ds) + \int_{A_*(\varepsilon)} \left\|y_0(s)\right\|^{\frac{p_*}{p}} r^{\frac{p - p_*}{p}} \mu(ds) \nonumber \\
&& \leq \frac{\varepsilon^2}{64d_*^2} \mu\left(A_*(\varepsilon)\right) + c_* \left(\frac{\varepsilon^2}{64 d_*^2} \right)^{\frac{p_*}{p}} \mu\left(A_*(\varepsilon)\right) \leq \frac{\varepsilon}{8d_*} \mu\left(\Omega\right) + c_* \left(\frac{\varepsilon^2}{64d_*^2} \right)^{\frac{1}{2}} \mu\left(\Omega\right) \nonumber \\
&& \leq \frac{\varepsilon}{8(\mu(\Omega)+1) (c_*+1)} \mu\left(\Omega\right) + c_* \frac{\varepsilon}{8(\mu(\Omega)+1) (c_*+1)} \mu\left(\Omega\right) \nonumber \\
&& \leq  \frac{\varepsilon}{8} + \frac{\varepsilon}{8} =\frac{\varepsilon}{4}
\end{eqnarray}

Choose an arbitrary $s\in A^*(\varepsilon).$ Then according to (\ref{eq331})  we have
\begin{eqnarray}\label{eq334}
1-\left(\frac{64rd_*^2}{\varepsilon^2}\right)^\frac{p-p_*}{p}<1-\left( \frac{r}{\left\|y_0(s)\right\|}\right)^\frac{p-p_*}{p} \leq 1-\left( \frac{r}{\alpha}\right)^\frac{p-p_*}{p} .
\end{eqnarray}

From inclusion $p\in (p_*,p_*+\gamma_2)$ and (\ref{eq323}) it follows that, $p-p_*<\gamma_2(\varepsilon,\alpha) \leq \sigma_4(\varepsilon, \alpha).$ Since $\alpha >\alpha_*(\varepsilon) \geq 2r,$ then applying (\ref{eq324}) it is possible to show that
\begin{eqnarray}\label{eq335}
1- \left( \frac{r}{\alpha}\right)^{\frac{p-p_*}{p}} < \frac{\varepsilon}{4\alpha \mu(\Omega)}.
\end{eqnarray}

Now, for $p\in (p_*,p_*+\gamma_2)$ we have that $p-p_*<\gamma_2(\varepsilon,\alpha) \leq \sigma_5(\varepsilon, \alpha).$ Taking into consideration (\ref{eq325}), the inequality $\varepsilon <\sigma_* \leq 4d_*\sqrt{r},$ it is not difficult to verify that
\begin{eqnarray}\label{eq336}
- \frac{\varepsilon}{4\alpha\mu(\Omega)} <1-\left(\frac{64rd_*^2}{ \varepsilon^2}\right)^{\frac{p-p_*}{p}}.
\end{eqnarray}

The inclusion $y_0(\cdot)\in B_{\Omega,\mathcal{X},p_*}^{\alpha}(r),$ (\ref{eq334}), (\ref{eq335}) and (\ref{eq336}) imply that
\begin{eqnarray}\label{eq339}
\int_{A^*(\varepsilon)} \left\|y_0(s)- y_0(s)\left\|y_0(s)\right\|^{\frac{p_* -p}{p}} r^{\frac{p - p_*}{p}}\right\| \mu(ds) &=& \int_{A^*(\varepsilon)}\left\|y_0(s)\right\|\left|1-\left( \frac{r}{\left\|y_0(s)\right\|}\right)^\frac{p-p_*}{p}\right| \mu(ds) \nonumber \\  &\leq &  \alpha \frac{\varepsilon}{4\alpha \mu(\Omega)} \mu(A^*(\varepsilon))\leq \frac{\varepsilon}{4}
\end{eqnarray} for $p\in (p_*,p_*+\gamma_2).$ Finally, (\ref{eq332}), (\ref{eq332a}) and (\ref{eq339}) yield that
\begin{eqnarray*}
\left\|y_0(\cdot)-y_p(\cdot)\right\|_1 \leq  \frac{\varepsilon}{4} + \frac{\varepsilon}{4}=\frac{\varepsilon}{2}
\end{eqnarray*} for  $p\in (p_*,p_*+\gamma_2).$

Since $p\in (p_*,p_*+\gamma_2),$ $y_0(\cdot)\in B_{\Omega,\mathcal{X},p_*}^{\alpha}(r)$ are arbitrarily chosen and $y_p(\cdot)\in B_{\Omega,\mathcal{X},p}^{\alpha}(r),$ we obtain the proof of the inclusion (\ref{eq322a}).

Let $\alpha(\varepsilon)=2\alpha_*(\varepsilon),$  $\delta_2(\varepsilon)=\gamma_2(\varepsilon,\alpha(\varepsilon))$ where $\gamma_2(\varepsilon,\alpha(\varepsilon))$ is defined  by (\ref{eq323}) as $\alpha=\alpha(\varepsilon)$ and  $\alpha_*(\varepsilon)>0$ is defined  by (\ref{alep}).
Then  according to (\ref{eq322a}), for every $p\in \left(p_*,p_*+\delta_2(\varepsilon)\right)$ the inclusion
\begin{eqnarray}\label{eq340}
B_{\Omega,\mathcal{X},p_*}^{\alpha(\varepsilon)}(r) \subset B_{\Omega,\mathcal{X},p}^{\alpha(\varepsilon)}(r) +\frac{\varepsilon}{2} B_{\Omega,\mathcal{X},1}(1)
\end{eqnarray} is held.

Proposition \ref{prop2.1} implies that for $\alpha(\varepsilon)$ the inclusions
\begin{eqnarray}\label{eq341}
B_{\Omega,\mathcal{X},p}(r) \subset B_{\Omega,\mathcal{X},p}^{\alpha(\varepsilon)}(r) +\frac{\varepsilon}{4} B_{\Omega,\mathcal{X},1}(1), \ \
B_{\Omega,\mathcal{X},p}^{\alpha(\varepsilon)}(r) \subset B_{\Omega,\mathcal{X},p}(r) +\frac{\varepsilon}{4} B_{\Omega,\mathcal{X},1}(1)
\end{eqnarray} are satisfied for every $p\in \left[\frac{p_*+1}{2},2p_*\right].$

Since $\delta_2(\varepsilon) \leq p_*,$ then (\ref{eq340}) and (\ref{eq341}) imply that
\begin{eqnarray*}
B_{\Omega,\mathcal{X},p_*}(r) \subset B_{\Omega,\mathcal{X},p_*}^{\alpha(\varepsilon)}(r) +\frac{\varepsilon}{4} B_{\Omega,\mathcal{X},1}(1) \subset B_{\Omega,\mathcal{X},p}^{\alpha(\varepsilon)}(r) +\frac{3\varepsilon}{4} B_{\Omega,\mathcal{X},1}(1) \subset B_{\Omega,\mathcal{X},p}(r) +\varepsilon B_{\Omega,\mathcal{X},1}(1)
\end{eqnarray*} for every $p\in \left(p_*,p_*+\delta_2(\varepsilon)\right).$
\end{proof}

From Proposition \ref{prop3.1} and Proposition \ref{prop3.2} it follows that the set valued map $p\rightarrow B_{\Omega,\mathcal{X},p}(r),$ $p\in \left[\frac{p_*+1}{2},2p_*\right],$ is lower semicontinuous at $p_*.$
Denote
\begin{eqnarray}\label{eq343}
\delta_*(\varepsilon)=\min \left\{\delta_1(\varepsilon), \delta_2(\varepsilon)\right\}
\end{eqnarray} where $\delta_1(\varepsilon)$ and $\delta_2(\varepsilon)$ are defined in  Proposition \ref{prop3.1} and Proposition \ref{prop3.2} respectively. Since $\delta_1(\varepsilon)\in \left(0,\frac{p_*-1}{2}\right],$ $\delta_2(\varepsilon)\in \left(0,p_*\right],$ then from (\ref{eq343}) we obtain that $\delta_*(\varepsilon)\in \left(0,\frac{p_*-1}{2}\right].$

\begin{theorem}\label{teo3.1}
Assume $p_*>1,$ $\varepsilon \in \left(0, \sigma_*\right).$ Then for every $p\in \left(p_*-\delta_*(\varepsilon),p_*+\delta_*(\varepsilon)\right)$ the inclusion
\begin{eqnarray*}
B_{\Omega,\mathcal{X},p_*}(r) \subset B_{\Omega,\mathcal{X},p}(r) +\varepsilon B_{\Omega,\mathcal{X},1}(1)
\end{eqnarray*} holds where $\sigma_* >0$ is defined by (\ref{sigma*}), $\delta_*(\varepsilon)\in \left(0,\frac{p_*-1}{2}\right]$ is defined by (\ref{eq343}).
\end{theorem}

\section{Upper Semicontinuty}

In this section it will be proved that the set valued map $p\rightarrow B_{\Omega,\mathcal{X},p}(r)$, $p>1,$ is upper semicontinuous.

\begin{proposition}\label{prop4.1}
Suppose $p_*>1,$ $\varepsilon \in \left(0, \sigma_*\right).$ Then there exists $\delta_3(\varepsilon) \in\left(0, \frac{p_*-1}{2}\right]$ such that for every $p\in \left(p_*, p_*-\delta_3(\varepsilon)\right)$ the inclusion
\begin{eqnarray*}
B_{\Omega,\mathcal{X},p}(r) \subset B_{\Omega,\mathcal{X},p_*}(r) +\varepsilon B_{\Omega,\mathcal{X},1}(1)
\end{eqnarray*} is satisfied where $\sigma_* >0$ is defined by (\ref{sigma*}).
\end{proposition}

\begin{proof}
Let us show that for $\varepsilon \in \left(0, \sigma_*\right)$ and $\alpha >\alpha_*(\varepsilon)$  there exists $\gamma_3=\gamma_3(\varepsilon,\alpha) \in\left(0, \frac{p_*-1}{2}\right]$ such that for every $p\in \left(p_*-\gamma_3, p_*\right)$ the inclusion
\begin{eqnarray}\label{eq344a}
B_{\Omega,\mathcal{X},p}^{\alpha}(r) \subset B_{\Omega,\mathcal{X},p_*}^{\alpha}(r) +\frac{\varepsilon}{2} B_{\Omega,\mathcal{X},1}(1)
\end{eqnarray} is satisfied where $\alpha_*(\varepsilon)>0$ is defined by (\ref{alep}).

We set
\begin{eqnarray}\label{eq344}
\gamma_3(\varepsilon,\alpha)=\min \left\{ \tau_1 (\varepsilon, \alpha), \tau_2 (\varepsilon, \alpha), \frac{p_*-1}{2} \right\}
\end{eqnarray}
where
\begin{eqnarray}\label{eq345}
\tau_1 (\varepsilon,\alpha)=p_* \log_{\displaystyle \frac{r}{\alpha}} \displaystyle \left(1-\frac{\varepsilon}{4\alpha \mu(\Omega)}\right),
\end{eqnarray}
\begin{eqnarray}\label{eq346}
\tau_2(\varepsilon,\alpha)=p_* \log_{\displaystyle \frac{64rd_*^2}{ \varepsilon^2}} \displaystyle \left(1+ \frac{\varepsilon}{4\alpha\mu(\Omega)}\right).
\end{eqnarray}

Since $\alpha>\alpha_*(\varepsilon)\geq 2r$ and $0<\varepsilon <\sigma_* \leq 4r\mu(\Omega),$ we have that $\tau_1(\varepsilon,\alpha)>0.$ Also, the inequality $\varepsilon <\sigma_* \leq 4\sqrt{r}d_*$ implies that $\frac{16rd_*^2}{\varepsilon^2}>1$ and consequently $\tau_2(\varepsilon,\alpha)>0.$ Thus, (\ref{eq344}) implies that $\gamma_3=\gamma_3(\varepsilon, \alpha)\in \left(0,\frac{p_*-1}{2}\right].$

Choose an arbitrary $p\in \left(p_*-\gamma_3, p_*\right)$ and fix it. Now let us choose an arbitrary $z_p(\cdot)\in B_{\Omega,\mathcal{X},p}^{\alpha}(r)$ and define new function $z_*(\cdot):\Omega\rightarrow \mathcal{X}$ setting
\begin{eqnarray}\label{eq347}
z_*(s)= z_p(s) \left\| z_p(s)\right\|^{\frac{p-p_*}{p_*}} r^{\frac{p_*-p}{p_*}}, \ \ s\in \Omega.
\end{eqnarray}

One can to specify that  $\left\|z_*(\cdot)\right\|_{p_*} \leq r.$
Now, since $p<p_*,$ $\alpha >\alpha_*(\varepsilon)\geq 2r,$ then from (\ref{eq347}) and inclusion $z_p(\cdot)\in B_{\Omega,\mathcal{X},p}^{\alpha}(r)$ we have
\begin{eqnarray*}\label{eq349}
\left\|z_*(s)\right\|= r\cdot \left(\frac{\left\|z_p(s)\right\|}{r}\right)^{\frac{p}{p_*}} \leq r\cdot \left(\frac{\alpha}{r}\right)^{\frac{p}{p_*}} \leq r\cdot \frac{\alpha}{r}=\alpha
\end{eqnarray*}
which implies that $z_*(\cdot)\in B_{\Omega,\mathcal{X},p_*}^{\alpha}(r).$

Denote
\begin{eqnarray}\label{eq350}
B_*(\varepsilon)=\left\{ s\in \Omega: 0\leq \left\| z_p(s) \right\| \leq  \left(\frac{\varepsilon}{8d_{*}}\right)^2\right\},
\end{eqnarray}
\begin{eqnarray}\label{eq351}
B^*(\varepsilon)=\left\{ s \in \Omega: \left(\frac{\varepsilon}{8d_{*}}\right)^{2} < \left\| z_p(s) \right\| \leq  \alpha\right\}
\end{eqnarray} where $d_*$ is defined by (\ref{de*}).

Since $\varepsilon <\sigma_*\leq 4d_*,$ $\displaystyle p\in \left(p_*-\gamma_3,p_*\right) \subset \left[ \frac{p_*+1}{2}, 2p_* \right]$, then we have  $\displaystyle \frac{\varepsilon}{8d_{*}}<1,$ $\displaystyle \frac{p}{p_{*}}>\frac{1}{2}.$ Thus, from (\ref{ce*}), (\ref{de*}) and (\ref{eq350}) it follows that
\begin{eqnarray}\label{eq353}
&& \int_{B_*(\varepsilon)} \left\|z_p(s)- z_p(s)\left\|z_p(s)\right\|^{\frac{p-p_*}{p_*}} r^{\frac{p_* - p}{p_*}}\right\| \mu(ds) \nonumber \\ && \leq  \int_{B_*(\varepsilon)} \left\|z_p(s)\right\| \mu(ds) + \int_{B_*(\varepsilon)}\left\|z_p(s)\right\|^{\frac{p}{p_*}} r^{\frac{p_*- p}{p_*}} \mu(ds) \nonumber \\ && \leq  \left(\frac{\varepsilon}{8d_{*}}\right)^2\mu(B_*(\varepsilon))+ c_* \left[\left(\frac{\varepsilon}{8d_{*}}\right)^2\right]^{\frac{p}{p_*}} \mu(B_*(\varepsilon))
\leq  \frac{\varepsilon}{8d_{*}}\mu(B_*(\varepsilon))+ c_* \left[\left(\frac{\varepsilon}{8d_{*}}\right)^2\right]^{\frac{1}{2}} \mu(B_*(\varepsilon))
\nonumber \\
&& = \frac{\varepsilon}{8(\mu(\Omega)+1)(c_{*}+1)}\mu(B_*(\varepsilon))+ c_* \frac{\varepsilon}{8(\mu(\Omega+1)(c_{*}+1)} \mu(B_*(\varepsilon)) < \frac{\varepsilon}{8}+ \frac{\varepsilon}{8}=\frac{\varepsilon}{4}.
\end{eqnarray}

Let us choose an arbitrary $s\in B^*(\varepsilon).$ Since $p\in (p_*-\gamma_3,p_*)$, then $\displaystyle \frac{p_* - p}{p_*}>0.$ Now, (\ref{eq351}) implies that
\begin{eqnarray}\label{eq356}
1-\left(\frac{64rd_{*}^2}{\varepsilon^2}\right)^{\frac{p_* - p}{p_*}} < 1-\left(\frac{r}{\left\| z_p(s) \right\|}\right)^{\frac{p_* - p}{p_*}} \leq  1-\left(\frac{r}{\alpha}\right)^{\frac{p_* - p}{p_*}}.
\end{eqnarray}

Taking into consideration that $p_*-p<\gamma_3 \leq \tau_1(\varepsilon,\alpha),$ $\alpha >\alpha_*(\varepsilon)\geq 2r$ and equality (\ref{eq345}) one can verify that
\begin{eqnarray}\label{eq357}
 1-\left(\frac{r}{\alpha}\right)^{\frac{p_* - p}{p_*}}<\frac{\varepsilon}{4\alpha \mu(\Omega)},
\end{eqnarray}

Similarly, since $p_*-p<\gamma_3 \leq \tau_2(\varepsilon,\alpha),$ $\displaystyle \varepsilon <\sigma_* \leq 4d_*\sqrt{r},$ then applying  (\ref{eq346}) we obtain that
\begin{eqnarray}\label{eq358}
-\frac{\varepsilon}{4\alpha \mu(\Omega)}< 1-\left(\frac{64rd_{*}^2}{\varepsilon^2}\right)^{\frac{p_* - p}{p_*}}.
\end{eqnarray}

Since $s\in B^*(\varepsilon)$ is an arbitrarily fixed, $z_p(\cdot)\in B_{\Omega,\mathcal{X},p}^{\alpha}(r)$, we have from (\ref{eq356}), (\ref{eq357}) and (\ref{eq358}) that
\begin{eqnarray}\label{eq360}
\int_{B^*(\varepsilon)} \left\|z_p(s)- z_p(s)\left\|z_p(s)\right\|^{\frac{p-p_*}{p_*}} r^{\frac{p_* - p}{p_*}}\right\| \mu(ds) &=&\int_{B^*(\varepsilon)} \left\|z_p(s)\right\|\left| 1-\left(\frac{r}{\left\| z_p(s) \right\|}\right)^{\frac{p_* - p}{p_*}} \right| \mu(ds) \nonumber \\ & \leq & \alpha \frac{\varepsilon}{4\alpha \mu (\Omega)} \mu(B^*(\varepsilon)) \leq \frac{\varepsilon}{4}.
\end{eqnarray}

On behalf of (\ref{eq350}), (\ref{eq351}), (\ref{eq353}) and (\ref{eq360}) we conclude
\begin{eqnarray}\label{eq361}
\left\|z_p(\cdot)-z_*(\cdot)\right\|_1 \leq \frac{\varepsilon}{4}+\frac{\varepsilon}{4}=\frac{\varepsilon}{2} .
\end{eqnarray}

Since $z_p(\cdot)\in B_{\Omega,\mathcal{X},p}^{\alpha}(r)$ is an arbitrarily chosen function,  $z_*(\cdot)\in B_{\Omega,\mathcal{X},p_*}^{\alpha}(r)$ then the inequality (\ref{eq361}) gives us the proof of the inclusion (\ref{eq344a}).

Let $\alpha(\varepsilon)=2\alpha_*(\varepsilon)$ where $\alpha_*(\varepsilon)>0$ is defined  by (\ref{alep}). By virtue of (\ref{eq344a}) we obtain that for $\delta_3(\varepsilon)=\gamma_3(\varepsilon,\alpha(\varepsilon)) \in\left(0, \frac{p_*-1}{2}\right]$ the inclusion
\begin{eqnarray}\label{eq362}
B_{\Omega,\mathcal{X},p}^{\alpha(\varepsilon)}(r) \subset B_{\Omega,\mathcal{X},p_*}^{\alpha(\varepsilon)}(r) +\frac{\varepsilon}{2} B_{\Omega,1}(1)
\end{eqnarray} is satisfied for every $p\in \left(p_*, p_*-\delta_3(\varepsilon)\right)$ where $\gamma_3(\varepsilon,\alpha(\varepsilon))>0$  is defined by (\ref{eq344}).

Now, on behalf of Proposition \ref{prop2.1} we have that for $\alpha(\varepsilon)$ the inclusions
\begin{eqnarray}\label{eq363}
B_{\Omega,\mathcal{X},p}(r) \subset B_{\Omega,\mathcal{X},p}^{\alpha(\varepsilon)}(r) +\frac{\varepsilon}{4} B_{\Omega,\mathcal{X},1}(1), \ \
B_{\Omega,\mathcal{X},p}^{\alpha(\varepsilon)}(r) \subset B_{\Omega,\mathcal{X},p}(r) +\frac{\varepsilon}{4} B_{\Omega,\mathcal{X},1}(1)
\end{eqnarray} are satisfied for every $p\in \left[\frac{p_*+1}{2},2p_*\right].$

From (\ref{eq362}) and (\ref{eq363})  it follows that for every $p\in \left(p_*, p_*-\delta_3(\varepsilon)\right)$ the inclusions
\begin{eqnarray*}
B_{\Omega,\mathcal{X},p}(r) \subset B_{\Omega,\mathcal{X},p}^{\alpha(\varepsilon)}(r) +\frac{\varepsilon}{4} B_{\Omega,\mathcal{X},1}(1) \subset
B_{\Omega,\mathcal{X},p_*}^{\alpha(\varepsilon)}(r) +\frac{3\varepsilon}{4} B_{\Omega,\mathcal{X},1}(1) \subset B_{\Omega,\mathcal{X},p_*}(r) +\varepsilon B_{\Omega,\mathcal{X},1}(1).
\end{eqnarray*} are verified.

The proof is completed.
\end{proof}

\begin{proposition}\label{prop4.2}
Assume $p_*>1,$ $\varepsilon \in \left(0, \sigma_*\right).$  Then there exists $\delta_4(\varepsilon) \in\left(0, p_*\right]$ such that for every $p\in \left(p_*,p_*+\delta_4 (\varepsilon) \right)$ the inclusion
\begin{eqnarray*}
B_{\Omega,\mathcal{X},p}(r) \subset B_{\Omega,\mathcal{X},p_*}(r) +\varepsilon B_{\Omega,\mathcal{X},1}(1)
\end{eqnarray*} is verified where $\sigma_* >0$ is defined by (\ref{sigma*}).
\end{proposition}

\begin{proof}
At the beginning let us show that for $\varepsilon \in \left(0, \sigma_*\right)$ and $\alpha_2 >\alpha_1 >\alpha_*(\varepsilon)$  there exists $\gamma_4=\gamma_4(\varepsilon,\alpha_1,\alpha_2) \in\left(0, p_*\right]$ such that for every $p\in \left(p_*,p_*+\gamma_4 \right)$ the inclusion
\begin{eqnarray}\label{eq365a}
B_{\Omega,\mathcal{X},p}^{\alpha_1}(r) \subset B_{\Omega,\mathcal{X},p_*}^{\alpha_2}(r) +\frac{\varepsilon}{2} B_{\Omega,\mathcal{X},1}(1)
\end{eqnarray} is satisfied where $\alpha_*(\varepsilon)>0$ is defined by (\ref{alep}).

Denote
\begin{eqnarray}\label{eq365}
\gamma_4(\varepsilon,\alpha_1,\alpha_2)=\min \left\{ \tau_3 (\alpha_1, \alpha_2), \tau_4 (\varepsilon, \alpha_1), \tau_5 (\varepsilon, \alpha_1), p_* \right\}
\end{eqnarray}
where
\begin{eqnarray}\label{eq366}
\tau_3 (\alpha_1, \alpha_2)=p_*\log_{\displaystyle \frac{\alpha_1}{r}}\frac{\alpha_2}{\alpha_1}
\end{eqnarray}
\begin{eqnarray}\label{eq367}
\tau_4 (\varepsilon, \alpha_1)= p_* \log_{\displaystyle \frac{\varepsilon}{4r\mu(\Omega)}} \left(1-\frac{\varepsilon}{4\alpha_1\mu(\Omega)}\right)
\end{eqnarray}
\begin{eqnarray}\label{eq368}
 \tau_5 (\varepsilon, \alpha_1) = p_* \log_{\displaystyle \frac{\alpha_1}{r}} \left(1+\frac{\varepsilon}{4\alpha_1\mu(\Omega)}\right).
\end{eqnarray}

Since $\alpha_2>\alpha_1 >\alpha_*(\varepsilon)\geq 2r,$ we have that $ \tau_3 (\alpha_1, \alpha_2)>0.$ From inequalities $\varepsilon <\sigma_*\leq 4r\mu(\Omega),$ $\alpha_1 >2r$ it follows that $\tau_4 (\varepsilon, \alpha_1)>0$ and $\tau_5 (\varepsilon, \alpha_1)>0.$ Thus, from (\ref{eq365}) we conclude that $\gamma_4(\varepsilon,\alpha_1,\alpha_2)\in (0,p_*].$

Choose an arbitrary $p\in (p_*,p_*+\gamma_4)$ and fix it. Now let us choose an arbitrary $v_p(\cdot)\in B_{\Omega,\mathcal{X},p}^{\alpha_1}(r)$ and define new function setting
\begin{eqnarray}\label{eq369}
v_*(s)= v_p(s) \left\| v_p(s)\right\|^{\frac{p-p_*}{p_*}} r^{\frac{p_*-p}{p_*}}, \ \ s\in \Omega.
\end{eqnarray}

It is possible to verify that $\left\|v_*(\cdot)\right\|_{p_*}\leq r.$

The inequalities  $\alpha_2>\alpha_1 >\alpha_*(\varepsilon)\geq 2r,$ $0<p-p_* < \gamma_4 \leq \tau_3(\alpha_1,\alpha_2)$ and (\ref{eq365}) imply that
\begin{eqnarray*}
\left(\frac{\alpha_1}{r}\right)^{\frac{p}{p_*}} < \frac{\alpha_2}{r}.
\end{eqnarray*}

Since $v_p(\cdot)\in B_{\Omega,\mathcal{X},p}^{\alpha_1}(r),$ then (\ref{eq369}) and the last inequality yield that
\begin{eqnarray*}
\left\|v_*(s)\right\|= r\left(\frac{\left\| v_p(s)\right\|}{r}\right)^{\frac{p}{p_*}} \leq r\left(\frac{\alpha_1}{r}\right)^{\frac{p}{p_*}}
 < r\cdot \frac{\alpha_2}{r}=\alpha_2.
\end{eqnarray*}

So, we obtain  $v_*(\cdot)\in B_{\Omega,\mathcal{X},p_*}^{\alpha_2}(r).$

According to (\ref{eq369}) we have that
\begin{eqnarray}\label{eq373}
\left\|v_*(\cdot)-v_p(\cdot)\right\|_1 =\int_{\Omega} \left\|v_p(s)\right\|\cdot \left|1- \left(\frac{\left\|v_p(s)\right\|}{r}\right)^{\frac{p-p_*}{p_*}} \right| \mu(ds) .
\end{eqnarray}

Let us set
\begin{eqnarray}\label{eq374}
C_*(\varepsilon)=\left\{ s\in \Omega: 0\leq \left\| v_p(s) \right\| \leq  \frac{\varepsilon}{4\mu(\Omega)}\right\}, \  \ \ C^*(\varepsilon)=\left\{ s \in \Omega: \frac{\varepsilon}{4\mu(\Omega)} < \left\| v_p(s) \right\| \leq  \alpha_1\right\}.
\end{eqnarray}

Choose an arbitrary $s\in C_*(\varepsilon).$ Since $\varepsilon \in (0,\sigma_*),$ $\sigma_* \leq 4r\mu(\Omega),$ $\frac{p-p_*}{p_*}>0,$ then \begin{eqnarray*}
0\leq 1- \left(\frac{\left\|v_p(s)\right\|}{r}\right)^{\frac{p-p_*}{p_*}} \leq 1 \ ,
\end{eqnarray*}
and hence
\begin{eqnarray}\label{eq376}
\int_{C_*(\varepsilon)} \left\|v_p(s)\right\|\cdot \left|1- \left(\frac{\left\|v_p(s)\right\|}{r}\right)^{\frac{p-p_*}{p_*}}\right| ds \leq \frac{\varepsilon}{4\mu(\Omega)} \cdot \mu(C_*(\varepsilon))\leq \frac{\varepsilon}{4}.
\end{eqnarray}

Now let us assume that $s\in C^*(\varepsilon).$ Then
\begin{eqnarray}\label{eq377}
1- \left(\frac{\alpha_1}{r}\right)^{\frac{p-p_*}{p_*}}\leq 1- \left(\frac{\left\|v_p(s)\right\|}{r}\right)^{\frac{p-p_*}{p_*}} \leq 1- \left(\frac{\varepsilon}{4\mu (\Omega)r}\right)^{\frac{p-p_*}{p_*}}
\end{eqnarray} for every $s\in C^*(\varepsilon).$

From inequalities $p-p_*<\gamma_4(\varepsilon,\alpha_1,\alpha_2))\leq \tau_4\left(\varepsilon,\alpha_1\right),$ $\varepsilon <\sigma_* \leq 4r\mu(\Omega)$ and (\ref{eq367}) it follows that
\begin{eqnarray}\label{eq378}
1- \left(\frac{\varepsilon}{4\mu (\Omega)r}\right)^{\frac{p-p_*}{p_*}} < \frac{\varepsilon}{4\mu(\Omega)\alpha_1}.
\end{eqnarray}

Now similarly, from the inequalities $\alpha_1 > \alpha_*(\varepsilon)\geq 2r,$ $p-p_*<\gamma_4 \leq \tau_5\left(\varepsilon,\alpha_1\right)$ and (\ref{eq368}) it is possible to obtain the validity of the inequality
\begin{eqnarray}\label{eq379}
-\frac{\varepsilon}{4\mu(\Omega)\alpha_1} <  1- \left(\frac{\alpha_1}{r}\right)^{\frac{p-p_*}{p_*}}.
\end{eqnarray}

By virtue of inclusion $v_p(\cdot)\in B_{\Omega,p}^{\alpha_1}(r),$ (\ref{eq374}), (\ref{eq377}), (\ref{eq378}) and (\ref{eq379}) we have
\begin{eqnarray}\label{eq381}
\int_{C^*(\varepsilon)} \left\|v_p(s)\right\|\cdot \left|1- \left(\frac{\left\|v_p(s)\right\|}{r}\right)^{\frac{p-p_*}{p_*}}\right| ds \leq \alpha_1 \frac{\varepsilon}{4\alpha_1 \mu(\Omega)} \cdot \mu(C^*(\varepsilon))\leq \frac{\varepsilon}{4}.
\end{eqnarray}

Finally, (\ref{eq373}), (\ref{eq374}), (\ref{eq376}) and (\ref{eq381}) yield
\begin{eqnarray}\label{eq382}
\left\|v_*(\cdot)-v_p(\cdot)\right\|_1 \leq  \frac{\varepsilon}{4}+\frac{\varepsilon}{4}=\frac{\varepsilon}{2}
\end{eqnarray}

Since $p\in (p_*,p_*+\gamma_4),$ $v_p(\cdot)\in B_{\Omega,\mathcal{X},p}^{\alpha_1}(r)$ are arbitrarily chosen, $v_*(\cdot)\in B_{\Omega,\mathcal{X},p}^{\alpha_2}(r),$ then  (\ref{eq382}) completes the proof of the inclusion (\ref{eq365a}).

Let $\alpha_1(\varepsilon)=2\alpha_*(\varepsilon),$ $\alpha_2(\varepsilon)=3\alpha_*(\varepsilon),$ $\delta_4(\varepsilon)=\gamma_4(\varepsilon,\alpha_1(\varepsilon),\alpha_2(\varepsilon))$ where $\gamma_4(\varepsilon,\alpha_1(\varepsilon),\alpha_2(\varepsilon))$ is defined  by (\ref{eq365}) as $\alpha_1=\alpha_1(\varepsilon)$, $\alpha_2=\alpha_2(\varepsilon)$ and $\alpha_*(\varepsilon)>0$ is defined  by (\ref{alep}).
Then  according to  (\ref{eq365a}) we have that for every $p\in \left(p_*,p_*+\delta_4(\varepsilon)\right)$ the inclusion
\begin{eqnarray}\label{eq383}
B_{\Omega,\mathcal{X},p}^{\alpha_1(\varepsilon)}(r) \subset B_{\Omega,\mathcal{X},p_*}^{\alpha_2(\varepsilon)}(r) +\frac{\varepsilon}{2} B_{\Omega,\mathcal{X},1}(1)
\end{eqnarray} is satisfied.

By virtue of Propositions \ref{prop2.1} for $\alpha_1(\varepsilon)=2\alpha_*(\varepsilon)$ and $\alpha_2(\varepsilon)=3\alpha_*(\varepsilon)$  the inclusions
\begin{eqnarray}\label{eq384}
B_{\Omega,\mathcal{X},p}(r) \subset B_{\Omega,\mathcal{X},p}^{\alpha_1(\varepsilon)}(r) +\frac{\varepsilon}{4} B_{\Omega,\mathcal{X},1}(1), \ \
B_{\Omega,\mathcal{X},p}^{\alpha_2(\varepsilon)}(r) \subset B_{\Omega,\mathcal{X},p}(r) +\frac{\varepsilon}{4} B_{\Omega,\mathcal{X},1}(1)
\end{eqnarray} are satisfied for every $p\in \left[\frac{p_*+1}{2},2p_*\right]$

Since $\delta_4(\varepsilon) \leq p_*,$ then (\ref{eq383}) and (\ref{eq384}) imply that
\begin{eqnarray*}
B_{\Omega,\mathcal{X},p}(r) \subset B_{\Omega,\mathcal{X},p}^{\alpha_1(\varepsilon)}(r) +\frac{\varepsilon}{4} B_{\Omega,\mathcal{X},1}(1) \subset B_{\Omega,\mathcal{X},p_*}^{\alpha_2(\varepsilon)}(r) +\frac{3\varepsilon}{4} B_{\Omega,\mathcal{X},1}(1) \subset B_{\Omega,\mathcal{X},p_*}(r) +\varepsilon B_{\Omega,\mathcal{X},1}(1)
\end{eqnarray*} for every $p\in \left(p_*,p_*+\delta_4(\varepsilon)\right).$
\end{proof}

From Proposition \ref{prop4.1} and Proposition \ref{prop4.2} we obtain that the set valued map $p\rightarrow B_{\Omega,\mathcal{X},p}(r),$ $p\in \left[\frac{p_*+1}{2},2p_*\right],$ is upper semicontinuous at $p_*.$
Let us set
\begin{eqnarray}\label{eq386}
\delta^*(\varepsilon)=\min \left\{\delta_3(\varepsilon), \delta_4(\varepsilon)\right\}
\end{eqnarray} where $\delta_3(\varepsilon)$ and $\delta_4(\varepsilon)$ are defined in  Proposition \ref{prop4.1} and Proposition \ref{prop4.2} respectively. Since $\delta_3(\varepsilon) \in\left(0, \frac{p_*-1}{2}\right],$ $\delta_4(\varepsilon) \in\left(0, p_*\right],$ then  (\ref{eq386}) imply that $\delta^*(\varepsilon) \in\left(0, \frac{p_*-1}{2}\right].$

\begin{theorem}\label{teo4.1}
Assume $p_*>1,$ $\varepsilon \in \left(0, \sigma_*\right).$ Then  for every $p\in \left(p_*-\delta^*(\varepsilon),p_*+\delta^*(\varepsilon)\right)$ the inclusion
\begin{eqnarray*}
B_{\Omega,\mathcal{X},p}(r) \subset B_{\Omega,\mathcal{X},p_*}(r) +\varepsilon B_{\Omega,\mathcal{X},1}(1)
\end{eqnarray*} is held where $\sigma_* >0$ is defined by (\ref{sigma*}), $\delta^*(\varepsilon) \in\left(0, \frac{p_*-1}{2}\right]$ is defined by (\ref{eq386}).
\end{theorem}

\section{Main Result}

From Theorem \ref{teo3.1} and Theorem \ref{teo4.1} we obtain the main result, the continuity of the  set valued map  $p\rightarrow B_{\Omega,\mathcal{X},p}(r),$ $p\in \left(1,+\infty\right).$

\begin{theorem}\label{teo5.1}
The set valued map  $p\rightarrow B_{\Omega,\mathcal{X},p}(r),$ $p\in \left(1,+\infty\right)$ is continuous, i.e. for each fixed $p_*>1$ and $\varepsilon \in \left(0, \sigma_*\right)$ there exists $\delta_0(\varepsilon)\in\left(0, \frac{p_*-1}{2}\right]$ such that for every $p\in \left(p_*-\delta_0(\varepsilon),p_*+\delta_0(\varepsilon)\right)$ the inequality
\begin{eqnarray*}
\mathcal{H}_1 \left(B_{\Omega,\mathcal{X},p}(r), B_{\Omega,\mathcal{X},p_*}(r)\right) \leq \varepsilon
\end{eqnarray*} is verified where $\sigma_* >0$ is defined by (\ref{sigma*}),
\begin{eqnarray*}
\delta_0(\varepsilon)=\min \left\{\delta_*(\varepsilon), \delta^*(\varepsilon)\right\},
\end{eqnarray*} $\delta_*(\varepsilon)\in\left(0, \frac{p_*-1}{2}\right]$ and $\delta^*(\varepsilon)\in\left(0, \frac{p_*-1}{2}\right]$ are defined in  Theorem \ref{teo3.1} and Theorem \ref{teo4.1} respectively, $\mathcal{H}_1(\cdot,\cdot)$ is defined by (\ref{haus}).
\end{theorem}

\begin{remark} Note that the numbers $\delta_*(\varepsilon),$ $\delta^*(\varepsilon)$ and hence the number $\delta_0(\varepsilon)$ also depend on $p_*.$ For simplicity of the expressions we do not write $p_*$ in their formulas.
\end{remark}

\section{Application}

Suppose $\left(\mathcal{X},\left\|\cdot \right\|\right)$ and $\left(\mathcal{Y},\left\|\cdot \right\|_{\mathcal{Y}}\right)$ are separable Banach spaces, $(\Omega, \Sigma, \mu)$ and $(\Omega_0, \Sigma_0, \mu_0)$ are finite and positive measure spaces, $(\Omega_*, \Sigma_*, \mu_*)$ is the product of the measure spaces $(\Omega, \Sigma, \mu)$ and $(\Omega_0, \Sigma_0, \mu_0).$

Consider input-output system described by Urysohn type integral operator
\begin{eqnarray}\label{eq61}
F(x(\cdot))|(\xi)=\int_{\Omega} K(\xi,s,x(s)) \mu (ds) \ \ \ \mbox{for $\mu_0$-almost all (a.a.)} \  \xi \in \Omega_0
\end{eqnarray} where $x(\cdot)\in B_{\Omega,\mathcal{X},p}(r)$ is the input, $F(x(\cdot))|(\cdot)$ is the output of the system, generated by the input $x(\cdot),$ $B_{\Omega,\mathcal{X},p}(r)$ is defined by (\ref{bep}).

The integral operators are applied for description of the behaviour of different type input-output systems. It should be underlined that the integral models have some advantages over differential ones, since the outputs for such systems can be defined as continuous, even as p-integrable functions. Note that the inputs $x(\cdot)\in B_{\Omega,\mathcal{X},p}(r),$ i.e. the integrally constrained inputs characterise the system inputs which are exhausted by consumption such as energy, fuel, finance etc. This kind of inputs are widely used in control systems theory where the system's control resource is limited (see, e.g. \cite{gusev}, \cite{kra}, \cite{rou}, \cite{sub}).

It is assumed that the function $K(\cdot,\cdot,\cdot):\Omega_0 \times \Omega \times \mathcal{X} \rightarrow \mathcal{Y}$ satisfies the following conditions:

\vspace{3mm}

6.A. The function $K(\cdot,\cdot,x):\Omega_0 \times \Omega  \rightarrow \mathcal{Y}$ \ is $\mu_*$-measurable for every $x\in \mathcal{X}$ and $K(\cdot,\cdot,0)\in L_{1}\left(\Omega_*, \Sigma_*, \mu_*; \mathcal{Y}\right)$.

\vspace{3mm}

6.B. There exists a function $\psi(\cdot,\cdot):\Omega_0\times \Omega \rightarrow [0,+\infty)$ such that $\psi (\xi,\cdot)\in L_{\infty}\left(\Omega,\Sigma,\mu;[0,\infty)\right)$ for $\mu_0$-a.a. $\xi \in \Omega_0,$ $\varphi (\cdot) \in L_{1}(\Omega_0,\Sigma_0,\mu_0;[0,\infty)),$ where $\varphi(\xi)=\left\|\psi (\xi,\cdot)\right\|_{\infty},$ and the inequality
\begin{eqnarray*}
\left\| K(\xi,s,x_1) -K(\xi,s,x_2)\right\|_{\mathcal{Y}} \leq \psi (\xi,s) \left\|x_1-x_2\right\|
\end{eqnarray*} is satisfied for every $x_1\in \mathcal{X}$ and $x_2\in \mathcal{X}$ and $\mu_*$-a.a. $(\xi,s)\in \Omega_0\times \Omega.$

Here $L_{\infty}\left(\Omega,\Sigma, \mu;[0,\infty)\right)$ is the space of all (equivalence classes of) $\mu$-measurable functions $v(\cdot):\Omega \rightarrow \left[0,\infty \right)$ such that $\left\|v(\cdot)\right\|_{\infty} <+\infty$ where $\left\|v(\cdot)\right\|_{\infty}=\inf \left\{ \rho>0: v(s)\leq  \rho \ \mbox{for $\mu$-a.a.} \ s \in \Omega \right\}.$

\vspace{3mm}

From conditions 6.A and 6.B it follows that for every input $x(\cdot)\in B_{\Omega,\mathcal{X},p}(r)$ the inclusion $F(x(\cdot))|(\cdot)\in  L_{1}\left(\Omega_0, \Sigma_0, \mu_0; \mathcal{Y}\right)$ is satisfied where $F(x(\cdot))|(\cdot)$ is the output of the system (\ref{eq61}) generated by the input $x(\cdot).$

Denote
\begin{eqnarray}\label{psi*k*}
\psi_*= \left\| \varphi(\cdot)\right\|_{1},  \  \ \ k_*=\left\| K(\cdot,\cdot,0)\right\|_{1} \, ,
\end{eqnarray}
\begin{eqnarray}\label{psi*0}
\theta_* =\max \left\{1, \mu(\Omega) \right\}, \ \ \psi_0= \psi_* \theta_* r+k_*.
\end{eqnarray}

We set
\begin{eqnarray}\label{eq62}
\mathcal{F}_{p}(r)= \left\{F(x(\cdot))|(\cdot): x(\cdot)\in B_{\Omega,\mathcal{X},p}(r) \right\}.
\end{eqnarray} The set $\mathcal{F}_{p}(r)$ is called the set of outputs of the system (\ref{eq61}). It is obvious that the set $\mathcal{F}_{p}(r)$ is the image of the closed ball $B_{\Omega,\mathcal{X},p}(r)\subset L_p(\Omega,\Sigma,\mu;\mathcal{X})$ under operator $F(\cdot)$ given by (\ref{eq61}). The approximate construction of the set of outputs of the input-output systems described by different type integral operators are considered in papers \cite{hus1}, \cite{hus2}, \cite{pol} (see the references also therein).

The set  $\mathcal{F}_{p}(r)$ is a bounded subset of the space $L_{1}(\Omega_0,\Sigma_0,\mu_0;\mathcal{Y})$ for every $p>1.$

\begin{proposition}\label{prop6.1}
For every $y(\cdot)\in \mathcal{F}_{p}(r)$ the inequality $\left\|y(\cdot)\right\|_{1} \leq \psi_0$ is satisfied where $\psi_0$ is defined by (\ref{psi*0}).
\end{proposition}

The proof of the proposition follows from conditions 6.A, 6.B and Fubini's theorem (see, \cite{war}, Theorem 1.4.45).

The following proposition asserts that the set valued map $p\rightarrow \mathcal{F}_{p}(r):(1,+\infty) \rightarrow L_{1}\left(\Omega_0,\Sigma_0,\mu_0;\mathcal{Y}\right)$ is continuous where the set $\mathcal{F}_{p}(r)$ is defined by (\ref{eq62}).
\begin{theorem}\label{teo6.1}
Let $p_*>1.$ Then  $\mathcal{H}_1\left(\mathcal{F}_{p}(r),\mathcal{F}_{p_*}(r)\right)\rightarrow 0$ as $p\rightarrow p_*.$
\end{theorem}
\begin{proof}
Let $\varepsilon >0$ be fixed number. According to the Theorem \ref{teo5.1} we have that for every $\varepsilon >0$ there exists $\delta_0=\delta_0(\varepsilon, p_*) \in \left(0,\frac{p_*-1}{2}\right]$ such that for every $p\in \left(p_*-\delta_0, p_*+\delta_0\right)$ the inequality
\begin{eqnarray}\label{eq65}
\mathcal{H}_1\left( B_{\Omega,\mathcal{X},p}(r), B_{\Omega,\mathcal{X},p_*}(r)\right)\leq \varepsilon
\end{eqnarray} is held.

Let us choose arbitrary $p\in \left(p_*-\delta_0, p_*+\delta_0\right)$ and $y(\cdot)\in \mathcal{F}_{p}(r)$ which is output of the system (\ref{eq61}) generated by the input $x(\cdot)\in B_{\Omega,\mathcal{X},p}(r).$
From (\ref{eq65}) it follows that for $x(\cdot)\in B_{\Omega,\mathcal{X},p}(r)$ there exists an input $x_*(\cdot)\in B_{\Omega,\mathcal{X},p_*}(r)$ such that
\begin{eqnarray}\label{eq67}
\left\|x(\cdot)-x_*(\cdot)\right\|_1 \leq \varepsilon.
\end{eqnarray}

Now let $y_*(\cdot)\in \mathcal{F}_{p_*}(r)$ be the output of the system (\ref{eq61}) generated by the input $x_*(\cdot)\in B_{\Omega,\mathcal{X},p}(r).$ From condition 6.B and (\ref{eq67}) it follows that
\begin{eqnarray*}
\left\| y(\xi)-y_*(\xi)\right\|_{\mathcal{Y}} \leq \int_{\Omega} \psi(\xi,s)\left\|x(s)-x_*(s)\right\| \mu(ds) \leq \varphi(\xi)\left\|x(\cdot)-x_*(\cdot)\right\|_1  \leq \varphi(\xi) \varepsilon
\end{eqnarray*}for $\mu_0$-a.a.  $\xi \in \Omega_0$ and hence
\begin{eqnarray}\label{eq69}
\left\| y(\cdot)-y_*(\cdot)\right\|_1 \leq  \psi_* \varepsilon
\end{eqnarray} where $\psi_*$ is defined by (\ref{psi*k*}). Since $y(\cdot)\in \mathcal{F}_{p}(r)$ is arbitrarily chosen, $y_*(\cdot)\in \mathcal{F}_{p_*}(r)$, we obtain from (\ref{eq69}) that
\begin{eqnarray}\label{eq610}
\mathcal{F}_{p}(r) \subset \mathcal{F}_{p_*}(r)+\psi_* \varepsilon B_{\Omega_0,\mathcal{Y},1}(1)
\end{eqnarray} where
\begin{eqnarray*}
B_{\Omega_0, \mathcal{Y},1}(1)=\left\{ y(\cdot)\in L_1 \left(\Omega_0, \Sigma_0,\mu_0; \mathcal{Y}\right): \left\| y(\cdot) \right\|_1 \leq 1 \right\}.
\end{eqnarray*}

Analogously it is possible to show that
\begin{eqnarray}\label{eq611}
\mathcal{F}_{p_*}(r) \subset \mathcal{F}_{p}(r)+\psi_* \varepsilon B_{\Omega_0,\mathcal{Y},1}(1)
\end{eqnarray} for every $p\in \left(p_*-\delta_0, p_*+\delta_0\right).$ From (\ref{eq610}) and (\ref{eq611}) we have
\begin{eqnarray*}
\mathcal{H}_1 \left(\mathcal{F}_{p}(r), \mathcal{F}_{p_*}(r)\right) \leq \psi_* \varepsilon
\end{eqnarray*}for every $p\in \left(p_0-\delta_*, p_*+\delta_0\right).$

The proposition is proved.
\end{proof}

\vspace{2mm}

\begin{remark} If $\Omega \subset \mathbb{R}^{k},$ $\Omega_0 \subset \mathbb{R}^{k_0}$ are the compact sets, $\mu$ and $\mu_0$ are Lebesgue measures, $\mathcal{X}=\mathbb{R}^m,$ $\mathcal{Y}=\mathbb{R}^n,$ the function $K(\cdot, \cdot, \cdot):\Omega_0 \times \Omega \times \mathbb{R}^m \rightarrow \mathbb{R}^n$ is continuous and
\begin{eqnarray*}
\left\| K(\xi,s,x_1) -K(\xi,s,x_2)  \right\| \leq  l_1 \left\| x_1-x_2 \right\|
\end{eqnarray*}for every $(\xi,s,x_1)\in \Omega_0 \times \Omega \times \mathbb{R}^m$ and $(\xi,s,x_2)\in \Omega_0 \times \Omega \times \mathbb{R}^m$ $(l_1 \geq 0)$, then $F(x(\cdot))|(\cdot) \in C(\Omega_0;\mathbb{R}^n)$ for every $x(\cdot)\in B_{\Omega,\mathbb{R}^m,p}(r)$ and $p>1$ where $C(\Omega_0;\mathbb{R}^n)$ is the space of continuous functions $y(\cdot):\Omega_0 \rightarrow \mathbb{R}^n$ with norm $\left\|y(\cdot)\right\|_C=\max \left\{ \left\| y(\xi) \right\|: \xi \in \Omega_0\right\}.$ In this case for every fixed $p_*>1$ the convergence $H_C \left(  \mathcal{F}_{p}(r), \mathcal{F}_{p_*}(r)\right) \rightarrow 0$ as $p\rightarrow p_*$ is satisfied. Here $H_C (\cdot,\cdot)$ denotes the Hausdorff distance between the subsets of the space $C(\Omega_0;\mathbb{R}^n).$
\end{remark}

\end{document}